\documentclass[a4paper, 12pt]{amsart}

\newtheorem{theorem}{Theorem}[section]
\newtheorem{proposition}[theorem]{Proposition}
\newtheorem{lemma}[theorem]{Lemma}
\newtheorem{corollary}[theorem]{Corollary}

\theoremstyle{definition}
\newtheorem{definition}[theorem]{Definition}

\begin{document}

\date{July 29, 2010}

\title[Ergodicity and mixing]{Ergodicity and mixing of W*-dynamical systems in terms of joinings}

\author{Rocco Duvenhage}

\address{Department of Mathematics and Applied Mathematics (Current address: Department of Physics)\\
University of Pretoria\\
Pretoria 0002\\
South Africa}

\email{rocco.duvenhage@up.ac.za}

\subjclass[2000]{46L55}

\begin{abstract}
We study characterizations of ergodicity, weak mixing and strong
mixing of W*-dynamical systems in terms of joinings and subsystems
of such systems. Ergodic joinings and Ornstein's criterion for
strong mixing are also discussed in this context.
\end{abstract}

\maketitle

\section{Introduction}

In \cite{D} we studied joinings of W*-dynamical systems, and in
particular gave a characterization of ergodicity in terms of
joinings, similar to the measure theoretic case. In this paper we
continue to extend certain results regarding joinings of measure
theoretic dynamical systems to the noncommutative setting of
W*-dynamical systems. First we generalize the necessary condition
for ergodicity to arbitrary group actions, and also prove a similar
set of sufficient and necessary conditions for weak mixing in terms
of ergodic compact systems and discrete spectra (see Section 2).
Section 3 is devoted to an interesting (and known) class of examples
of W*-dynamical systems obtained from group von Neumann algebras of
discrete groups and their automorphisms, however we express our
results in the language of locally compact quantum groups. Next we
study ergodic joinings in Section 4. In Sections 2 and 4 we also
consider simple applications for the case where the group action is
that of a countable discrete amenable group, namely a weak ergodic
theorem and a Halmos-von Neumann type theorem respectively. In the
latter we make the rather strong assumption of asymptotic
abelianness ``in density''. The focus in this paper is on building
some general aspects of the theory of joinings of W*-dynamical
systems, and these applications are more for illustration of how
joinings can potentially be used rather than being important results
in themselves. In Section 5 we present a joining characterization of
strong mixing (for the special case where the acting group is
$\mathbb{Z}$), and use it to obtain a version of Ornstein's
criterion for strong mixing in the case of W*-dynamical systems.
Sections 2 and 3 differ from Sections 4 and 5 in the sense that in
the former subsystems of W*-dynamical systems play a central role
while in the latter they do not. At the same time Sections 4 and 5
just take initial steps in the respective topics, while the topics
in Sections 2 and 3 are more fully developed. Along the way we give
indications of further work that might be done.

We use the same basic definitions as in \cite{D}, and will again
refer to a W*-dynamical system simply as a ``dynamical system'', or
even just a ``system''. For convenience we summarize the essential
definitions used in \cite{D}: A \textit{dynamical system}
$\mathbf{A}=\left(A,\mu,\alpha\right)$ consists of a faithful normal
state $\mu$ on a $\sigma$-finite von Neumann algebra $A$, and a
representation $\alpha:G\rightarrow$ Aut$(A):g\mapsto \alpha_{g}$ of
an arbitrary group $G$ as $\ast$-automorphisms of $A$, such that
$\mu\circ\alpha_{g}=\mu$ for all $g$. We will call $\mathbf{A}$ an
\textit{identity system} if $\alpha_{g}=\iota_{A}$ for all $g$ where
$\iota_{A}:A\rightarrow A$ is the identity mapping, while we call it
\textit{trivial} if $A=\mathbb{C}1_{A}$ where $1_{A}$ (often denoted
simply as $1$) is the unit of $A$. In the rest of the paper the
symbols $\mathbf{A}$, $\mathbf{B}$ and $\mathbf{F}$ will denote
dynamical systems $\left(A,\mu,\alpha\right)$,
$\left(B,\nu,\beta\right)$ and $\left( F,\kappa,\varphi\right)$
respectively,\ all making use of actions of the same group $G$. A
\textit{joining} of $\mathbf{A}$ and $\mathbf{B}$ is a state
$\omega$ (i.e. a positive linear functional with $\omega (1)=1$) on
the algebraic tensor product $A\odot B$ such that $\omega\left(
a\otimes1_{B}\right)  =\mu(a)$, $\omega\left(1_{A}\otimes
b\right)=\nu(b)$ and
$\omega\circ\left(\alpha_{g}\odot\beta_{g}\right) =\omega$ for all
$a\in A$, $b\in B$ and $g\in G$. The set of all joinings of
$\mathbf{A}$ and $\mathbf{B}$ is denoted by $J\left(
\mathbf{A},\mathbf{B}\right)  $. We call $\mathbf{A}$
\textit{disjoint} from $\mathbf{B}$ when
$J\left(\mathbf{A},\mathbf{B}\right) =\left\{\mu\odot\nu\right\}$. A
dynamical system $\mathbf{A}$ is called \textit{ergodic} if its
\textit{fixed point algebra} $A_{\alpha}:=\left\{a\in
A:\alpha_{g}(a)=a\text{ for all }g\in G\right\}$ is trivial, i.e.
$A_{\alpha}=\mathbb{C}1_{A}$. We call $\mathbf{F}$ a
\textit{subsystem} of $\mathbf{A}$ if there exists an injective
unital $\ast$-homomorphism $h$ of $F$ onto a von Neumann subalgebra
of $A$ such that $\mu\circ h=\kappa$ and $\alpha_{g}\circ
h=h\circ\varphi_{g}$ for all $g\in G$. (In \cite{D} the terminology
``factor" instead of ``subsystem" was used.) If furthermore
$h:F\rightarrow A$ is surjective, then we say that $h$ is an
\textit{isomorphism} of dynamical systems, and the systems
$\mathbf{A}$ and $\mathbf{F}$ are \textit{isomorphic}.

Unlike \cite{D}, in this paper we will have occasion to use completions of the
algebraic tensor product. Even though $A$ and $B$ are von Neumann algebras, we
will encounter the maximal C*-algebraic tensor product $A\otimes_{m}B$ in
Sections 2, 4 and 5. In Section 3 we do use the von Neumann algebraic tensor
product, however in this case it is to handle locally compact quantum groups
and not directly related to joinings.

The work in this paper is of course strongly influenced by previous work on
joinings in measure theoretic ergodic theory which originates in Furstenberg's
work \cite{F67}. In this regard we mention that \cite{dlR} and \cite{G}, as
well as unpublished lecture notes by A. del Junco, served as very useful sources.

For example the joining obtained in \cite[Construction 3.4]{D}, and
which we will again use here, can be viewed (ignoring dynamics) as a
generalization of a diagonal measure $\triangle(Y\times
Z):=\rho(Y\cap Z)$ defined in terms of some measure $\rho$ on a
measurable space $X$ and where $Y,Z\subset X$. A noncommutative
version of a diagonal measure using essentially the same idea as our
construction of a joining appeared in \cite[Section 4]{Fid}.

Also keep in mind that the use of joinings in noncommutative dynamical systems
is not without precedent, as a special case of this idea (under the name
``stationary couplings'') is used in work on entropy \cite{ST}.

\section{Ergodicity and weak mixing}

We start by improving on the characterization of ergodicity given in
\cite{D}. In particular we prove a stronger version of \cite[Theorem
3.7]{D} using a simpler proof. We do this by using an approach given
in unpublished lecture notes by A. del Junco for the measure
theoretic case.

\begin{theorem}
A dynamical system $\mathbf{A}$ is ergodic if and only if it is
disjoint from all identity systems.
\end{theorem}

\begin{proof}
Suppose $\mathbf{A}$ is ergodic, and let $\mathbf{B}$ be any
identity system. Consider any $\omega\in J(\mathbf{A},\mathbf{B})$.
From this joining we obtain (see \cite[Construction 2.3 and
Proposition 2.4]{D}) a conditional expectation operator
$P_{\omega}:H_{\mu }\rightarrow H_{\nu}$ (i.e. $\left\langle
P_{\omega}x,y\right\rangle =\left\langle x,y\right\rangle $) such
that $U_{g}P_{\omega}^{\ast}=P_{\omega}^{\ast}V_{g}$, where
$\gamma_{\mu}:A\rightarrow H_{\mu}$ and $\gamma_{\nu }:B\rightarrow
H_{\nu}$ are the GNS constructions for $\left(A,\mu\right)$ and
$\left(  B,\nu\right)  $ respectively, $U$ and $V$ the corresponding
unitary representations of $\alpha$ and $\beta$ on the Hilbert
spaces
 $H_{\mu}$ and $H_{\nu}$
respectively, and we denote by $\Omega_{\omega}$ their common unit cyclic
vector (in the GNS Hilbert space obtained from $\omega$, which contains
$H_{\mu}$ and $H_{\nu}$). Therefore for any $b\in B$ we have $U_{g}P_{\omega
}^{\ast}\gamma_{\nu}(b)=P_{\omega}^{\ast}\gamma_{\nu}(b)$, since $\mathbf{B}$
is an identity system. But $\mathbf{A}$ is ergodic, hence by \cite[Theorem
4.3.20]{BR} the fixed point space of $U$ is $\mathbb{C}\Omega_{\omega}$, so
$P_{\omega}^{\ast}\gamma_{\nu}(b)=\left\langle \Omega_{\omega},P_{\omega
}^{\ast}\gamma_{\nu}(b)\right\rangle \Omega_{\omega}=\nu(b)\Omega_{\omega}$.
For any $a\in A$ it follows that
\[
\omega\left(  a\otimes b\right)  =\left\langle
\gamma_{\mu}(a^{\ast}),\gamma_{\nu}(b)\right\rangle =\left\langle
\gamma_{\mu}(a^{\ast}),P_{\omega}^{\ast}\gamma_{\nu}(b)\right\rangle
=\mu(a)\nu(b)
\]
hence $\omega=\mu\odot\nu$, which means that $\mathbf{A}$ is
disjoint from $\mathbf{B}$. The converse was proven in \cite[Theorem
3.3]{D} using a subsystem of $\mathbf{A}$.
\end{proof}

Before we move on to weak mixing, we give a simple application of Theorem 2.1,
namely we prove a weak ergodic theorem. The result itself is not that
interesting, but we do this to illustrate how joinings can in principle be
used to prove results that don't refer to joinings in their formulation (see
in particular Corollary 2.4). Again we follow the basic plan for the measure
theoretic case given in the unpublished lecture notes by del Junco.

\begin{definition}
For a dynamical system $\mathbf{A}$, consider the cyclic
representation $\left(H,\pi,\Omega\right)$ of $\left(A,\mu\right)$
obtained by the GNS construction. Set $\tilde{A}:=\pi (A)^{\prime}$,
define the state $\tilde{\mu}$ on $\tilde{A}$ by
$\tilde{\mu}(b):=\left\langle \Omega,b\Omega\right\rangle $, and let
the unital $\ast $-homomorphism $\delta:A\odot\tilde{A}\rightarrow
B(H)$ be defined by $\delta\left( a\otimes b\right) :=\pi(a)b$. The
state $\mu_{\bigtriangleup}$ on the unital $\ast$-algebra
$A\odot\tilde{A}$ defined by $\mu_{\bigtriangleup}(t):=\left\langle
\Omega,\delta(t)\Omega\right\rangle $ will be called the
\textit{diagonal state} for $\left(A,\mu\right)$.
\end{definition}

The state $\mu_{\bigtriangleup}$ is in fact a joining of
$\mathbf{A}$ and its ``mirror image'' $\mathbf{\tilde{A}}$
constructed on $\left(\tilde{A},\tilde {\mu}\right)$ defined above
by carrying $\alpha$ to $\tilde{A}$ using the natural
$\ast$-anti-isomorphism $a\mapsto Ja^{\ast}J$ where $J$ is the
modular conjugation associated with $\left(\pi(A),\Omega\right)$
(see \cite[Construction 3.4]{D}). But it is not this aspect of
$\mu_{\bigtriangleup}$ that will be used in the next proposition
(see Section 5 for further elaboration on the joining aspect).

\begin{proposition}
Let $\mathbf{A}$ be ergodic, with $G$ countable, discrete and
amenable, and consider any right F\o lner sequence
$\left(\Lambda_{n}\right) $ in $G$. We can extend the diagonal state
for $\left( A,\mu\right)  $ to a state $\mu_{\bigtriangleup}$ on the
maximal C*-algebraic tensor product $A\otimes_{m}\tilde{A}$, and
then
\[
\operatorname{w*-lim}_{n\rightarrow\infty}\frac{1}{\left|
\Lambda_{n}\right|}
\sum_{g\in\Lambda_{n}}\mu_{\bigtriangleup}\circ\left(\alpha_{g}\otimes
_{m}\iota_{\tilde{A}}\right) =\mu\otimes_{m}\tilde{\mu}
\]
where $\operatorname{w*-lim}$ denotes the weak* limit and
$\iota_{\tilde{A}}$ is the identity mapping on $\tilde{A}$.
\end{proposition}

\begin{proof}
We will make use of the identity system
$\mathbf{B}:=\left(\tilde{A},\tilde{\mu},\iota_{\tilde{A}}\right)$.
The maximal tensor product has the property that $\delta$ in
Definition 2.2 can be extended to a $\ast $-homomorphism
$A\otimes_{m}\tilde{A}\rightarrow B(H)$, and hence we can easily
extend the diagonal state to a state $\mu_{\bigtriangleup}$ on
$A\otimes_{m}\tilde{A}$. (The general case of such extensions is
discussed in Section 4.) Then
\[
\omega_{n}:=\frac{1}{\left|
\Lambda_{n}\right|}\sum_{g\in\Lambda_{n}}
\mu_{\bigtriangleup}\circ\left(\alpha_{g}\otimes_{m}\iota_{\tilde{A}}
\right)
\]
is also a state on $A\otimes_{m}\tilde{A}$. The set $S$ of states of the
unital C*-algebra $A\otimes_{m}\tilde{A}$ is weakly* compact (see for example
\cite[Theorem 2.3.15]{BR}), hence the sequence $\left(  \omega_{n}\right)  $
has a cluster point $\rho$ in $S$ in the weak* topology.

We now show that $\rho|_{A\odot\tilde{A}}$ is a joining of
$\mathbf{A}$ and $\mathbf{B}$. For each $\varepsilon>0$, $a\in A$,
$b\in\tilde{A}$ and $N\in\mathbb{N}$, there is an $n>N$ such that
$\left|  \rho\left( a\otimes b\right)  -\omega_{n}\left( a\otimes
b\right) \right| <\varepsilon$. Furthermore,
$\omega_{n}\left(a\otimes1_{\tilde{A}}\right) =\mu(a)$ and
$\omega_{n}\left( 1_{A}\otimes b\right)  =\tilde{\mu}(b)$. Therefore
$\left| \rho\left(  a\otimes1_{\tilde{A}}\right)  -\mu(a)\right|
<\varepsilon$ and $\left|  \rho\left(  1_{A}\otimes b\right)
-\tilde{\mu}(b)\right| <\varepsilon$ for all $\varepsilon>0$, and so
$\rho\left(  a\otimes 1_{\tilde{A}}\right)  =\mu(a)$ and
$\rho\left(1_{A}\otimes b\right) =\tilde{\mu}(b)$. Next note that
for all $h\in G$
\begin{align*}
&  \left|  \omega_{n}\circ\left(
\alpha_{h}\otimes_{m}\iota_{\tilde{A}} \right)  \left(  a\otimes
b\right)  -\omega_{n}\left(  a\otimes b)\right)
\right| \\
&  =\frac{1}{\left|  \Lambda_{n}\right|  }\left|  \sum_{g\in\left(
\Lambda_{n}h\right)  \backslash\Lambda_{n}}\mu_{\bigtriangleup}\circ\left(
\alpha_{g}\otimes_{m}\iota_{\tilde{A}}\right)  \left(  a\otimes b\right)
-\sum_{g\in\Lambda_{n}\backslash\left(  \Lambda_{n}h\right)  }\mu
_{\bigtriangleup}\circ\left(  \alpha_{g}\otimes_{m}\iota_{\tilde{A}}\right)
\left(  a\otimes b\right)  \right| \\
&  \leq\frac{\left|  \Lambda_{n}\bigtriangleup\left(  \Lambda_{n}h\right)
\right|  }{\left|  \Lambda_{n}\right|  }\left\|  a\otimes b\right\| \\
&  \rightarrow0
\end{align*}
as $n\rightarrow\infty$. Since $\rho$ is a cluster point of
$\left(\omega_{n}\right)$, we conclude that
$\rho\circ\left(\alpha_{g}\otimes
_{m}\iota_{\tilde{A}}\right)=\rho$, and therefore
$\rho|_{A\odot\tilde{A}}\in J\left(\mathbf{A},\mathbf{B}\right)$.

By Theorem 2.1 and continuity it follows that
$\rho=\mu\otimes_{m}\tilde{\mu}$. In particular this means that
$\mu\otimes_{m}\tilde{\mu}$ is the unique weak* cluster point of
$\left(\omega_{n}\right)$, which implies that
$\left(\omega_{n}\right)$ converges to $\mu\otimes_{m}\tilde{\mu}$,
as required.
\end{proof}

To clarify the meaning of Proposition 2.3, we include the following weak mean
ergodic theorem in terms of a Hilbert space (the conventional proof of the
mean ergodic theorem is both more elementary, and delivers a stronger result
than the current approach, but again, our motivation here is to illustrate
that results regarding joinings can have nontrivial consequences). This result
essentially turns the logic of the proof of \cite[Theorem 3.7]{D} around:

\begin{corollary}
Consider the situation in Definition 2.2 and Proposition 2.3, and
let $U$ be the unitary representation of $\alpha$ on $H$, in other
words $\pi\left(\alpha _{g}(a)\right) =U_{g}\pi(a)U_{g}^{\ast}$ and
$U_{g}\Omega=\Omega$. Then
\[
\lim_{n\rightarrow\infty}\frac{1}{\left|  \Lambda_{n}\right|  }\sum
_{g\in\Lambda_{n}}\left\langle U_{g}x,y\right\rangle =\left\langle \left(
\Omega\otimes\Omega\right)  x,y\right\rangle
\]
for all $x,y\in H$, where $\Omega\otimes\Omega$ is the projection of
$H$ onto $\mathbb{C}\Omega$, i.e. $\left( \Omega\otimes\Omega\right)
x=\Omega\left\langle \Omega,x\right\rangle$.
\end{corollary}

\begin{proof} For $x:=\pi(a)\Omega$ and $y:=b\Omega$ where $a\in A$
and $b\in\tilde{A}$, it follows from Proposition 2.3 that
\begin{align*}
\left\langle \left(  \Omega\otimes\Omega\right)  x,y\right\rangle  &
=\mu\otimes_{m}\tilde{\mu}\left(  a^{\ast}\otimes b\right) \\
&  =\lim_{n\rightarrow\infty}\frac{1}{\left|  \Lambda_{n}\right|  }\sum
_{g\in\Lambda_{n}}\mu_{\bigtriangleup}\left(  \alpha_{g}(a^{\ast})\otimes
b\right) \\
&  =\lim_{n\rightarrow\infty}\frac{1}{\left|  \Lambda_{n}\right|  }\sum
_{g\in\Lambda_{n}}\left\langle U_{g}x,y\right\rangle
\end{align*}
but $\pi(A)\Omega$ and $\tilde{A}\Omega$ are both dense in $H$, since $\mu$ is
faithful and normal.
\end{proof}

We now proceed to weak mixing, our goal being an analogue of Theorem 2.1.

\begin{definition}
Consider a dynamical system $\mathbf{A}$ and let
$\left(H,\pi,\Omega\right)$ be the cyclic representation of
$\left(A,\mu\right)$ obtained from the GNS construction, and let $U$
be thecorresponding unitary representation of $\alpha$ on $H$, i.e.
$U_{g} \pi(a)\Omega=\pi(\alpha_{g}(a))\Omega$. An
\textit{eigenvector} of $U$ is an $x\in H\backslash\{0\}$ such that
there is a function, called its \textit{eigenvalue},
$\chi_{x}:G\rightarrow\mathbb{C}$ such that $U_{g} x=\chi_{x}(g)x$
for all $g\in G$. The eigenvalue $g\mapsto1$ will be denoted as $1$.
Denote by $H_{0}$ the Hilbert subspace spanned by the eigenvectors
of $U$. The set of all eigenvalues is denoted by
$\sigma_{\mathbf{A}}$ and is called the \textit{point spectrum} of
$\mathbf{A}$. We call $\mathbf{A}$ \textit{weakly mixing} if $\dim
H_{0}=1$. We say $\mathbf{A}$ \textit{has discrete spectrum} if
$H_{0}=H$. We call $\mathbf{A}$ \textit{compact} if the orbit
$U_{G}x$ is totally bounded in $H$ for every $x\in H$, or,
equivalently, if $\alpha_{G}(a)$ is totally bounded in
$\left(A,\left\|  \cdot\right\| _{\mu}\right)$ for every $a\in A$,
where $\left\|  a\right\|_{\mu}:=\sqrt{\mu\left(a^{\ast}a\right)}$.
\end{definition}

We have the following equivalence when $G$ is abelian:

\begin{proposition}
Let $G$ be abelian. Then $\mathbf{A}$ has discrete spectrum if and
only if it is compact.
\end{proposition}

\begin{proof}
By \cite[Section 2.4]{K} (or see \cite[Lemma 6.6]{BDS} for the
special case that we are using here), $H_{0}$ is the set of all
$x\in H$ whose orbits $U_{G}x$ are totally bounded in $H$.
\end{proof}

It is not clear if Proposition 2.6 can be extended to nonabelian $G$.
Therefore we are going to give the sufficient and necessary conditions for
weak mixing separately in terms of compactness and discrete spectra respectively.

\begin{theorem}
Let $\mathbf{A}$ be ergodic. If $\mathbf{A}$ is disjoint from all
ergodic compact systems, then it is weakly mixing.
\end{theorem}

\begin{proof}
The plan is essentially the same as for the proof of the
corresponding direction in Theorem 2.1 (see \cite[Theorem 3.3]{D}).
Suppose $\mathbf{A}$ is not weakly mixing, then by
\cite[Propositions 6.5 and 6.7(1)]{BDS} it has a nontrivial compact
subsystem, say $\mathbf{F}$. Since $\mathbf{A} $ is ergodic, so is
$\mathbf{F}$. So by \cite[Construction 3.4 and Lemma 3.5]{D} we are
finished.
\end{proof}

\begin{theorem}
If $\mathbf{A}$ and $\mathbf{B}$ are ergodic, and $\mathbf{B}$ has
discrete spectrum with $\sigma
_{\mathbf{A}}\cap\sigma_{\mathbf{B}}=\{1\}$, then $\mathbf{A}$ is disjoint from
$\mathbf{B}$. In particular, if $\mathbf{A}$ is weakly
mixing, then it is disjoint from all ergodic systems with discrete
spectrum.
\end{theorem}

\begin{proof}
As in the proof of Theorem 2.1, we employ a conditional expectation
operator. So consider any $\omega\in
J\left(\mathbf{A},\mathbf{B}\right)$, and then use the same notation
as in Theorem 2.1's proof. Let $y\in H_{\nu}$ be any eigenvector of
$V$ with eigenvalue $\chi$, then $y=\gamma_{\nu}(e)$ for some $e\in
B$ by \cite[Theorem 2.5]{S}, while
$U_{g}P_{\omega}^{\ast}y=\chi(g)P_{\omega}^{\ast}y$. So either
$P_{\omega }^{\ast}y=0$ or
$\chi\in\sigma_{\mathbf{A}}\cap\sigma_{\mathbf{B}}$. In the latter
case $P_{\omega}^{\ast}y\in\mathbb{C}\Omega_{\omega}$, since
$\mathbf{A}$ is ergodic, hence in either case we have
$P_{\omega}^{\ast}y\in\mathbb{C}\Omega_{\omega}$. Therefore
\begin{align*}
\left\langle \gamma_{\mu}(a^{\ast}),\gamma_{\nu}(e)\right\rangle  &
=\left\langle \gamma_{\mu}(a^{\ast}),P_{\omega}^{\ast}\gamma_{\nu
}(e)\right\rangle =\left\langle \gamma_{\mu}(a^{\ast}),\Omega_{\omega
}\right\rangle \left\langle \Omega_{\omega},P_{\omega}^{\ast}\gamma_{\nu
}(e)\right\rangle \\
&  =\mu(a)\left\langle \Omega_{\omega},\gamma_{\nu}(e)\right\rangle
\end{align*}
for all $a\in A$. For an arbitrary $b\in B$ one has a sequence $\left(
b_{n}\right)  $ of linear combinations of such \textit{eigenoperators }$e$,
such that $\gamma_{\nu}(b_{n})\rightarrow\gamma_{\nu}(b)$, since $\mathbf{B}$
has discrete spectrum. Hence
\begin{align*}
\omega\left(  a\otimes b\right)   &  =\left\langle \gamma_{\mu}(a^{\ast
}),\gamma_{\nu}(b)\right\rangle =\lim_{n\rightarrow\infty}\left\langle
\gamma_{\mu}(a^{\ast}),\gamma_{\nu}(b_{n})\right\rangle =\lim_{n\rightarrow
\infty}\mu(a)\left\langle \Omega,\gamma_{\nu}(b_{n})\right\rangle \\
&  =\mu(a)\nu(b)
\end{align*}
which means that $J\left(  \mathbf{A},\mathbf{B}\right)  =\{\mu\odot\nu\}$.
\end{proof}

\section{The quantum group duals of discrete groups}

Halmos \cite{H} studied dynamical systems consisting of an automorphism of a
compact abelian group, with the automorphism providing an action of
$\mathbb{Z}$ on the group by iteration. In particular he characterized
ergodicity (which turns out to be equivalent to strong mixing in this case) in
terms of the orbits of the induced action in the dual group (or character
group). Here we study a generalization of this type of system, where the
compact group is replaced by a compact quantum group obtained as the dual of a
discrete group $\Gamma$ which need not be abelian. For simplicity we also only
consider actions of $G=\mathbb{Z}$ in this section.

We use the von Neumann algebra setting for locally compact quantum groups
(which include both our discrete group and its compact quantum group dual as
special cases), as developed by Kustermans and Vaes \cite{KV4} (also see
\cite{vD} and \cite{KV3}). Below we briefly review the definitions of this
theory to fix the conventions and notations that we will use. Other useful
sources regarding this material is \cite{VPhD}, and \cite[Section 18]{St}
which focusses on Hopf-von Neumann algebras and Kac algebras.

We should mention that since we are ultimately only interested in discrete
groups and their dual quantum groups, we could in principle work in the
setting of Kac algebras or even in terms of group von Neumann algebras.
However the framework set up in \cite{KV4} is simple and powerful, and very
convenient to work in, while the language of quantum groups also makes the
generalization from abelian to general discrete groups clearer, and opens the
window to possible further generalization when replacing the discrete group by
a discrete quantum group (which we will not do in this paper).

A \textit{locally compact quantum group} is defined to be a von Neumann
algebra $M$ with a unital normal $\ast$-homomorphism $\Delta:M\rightarrow
M\otimes M$ (where $M\otimes N$ denotes the von Neumann algebraic tensor
product of two von Neumann algebras), such that $(\Delta\otimes\iota_{M}%
)\circ\Delta=(\iota_{M}\otimes\Delta)\circ\Delta$ (where $\iota_{M}$ denotes
the identity map on $M)$, and on which there exist normal semi-finite faithful
(n.s.f.) weights $\varphi$ and $\psi$ which are left and right invariant
respectively, namely $\varphi\left(  (\theta\otimes\iota_{M})\circ
\Delta(a)\right)  =\varphi(a)\theta(1)$ for all $a\in\mathcal{M}_{\varphi}^{+}
$ and $\psi\left(  (\iota_{M}\otimes\theta)\circ\Delta(a)\right)
=\psi(a)\theta(1)$ for all $a\in\mathcal{M}_{\psi}^{+}$, for all $\theta\in
M_{\ast}^{+}$, where $M_{\ast}^{+}$ is the positive normal linear functionals
on $M$, and $\mathcal{M}_{\varphi}^{+}=\left\{  a\in M^{+}:\varphi
(a)<\infty\right\}  $. This quantum group is denoted as $\left(
M,\Delta\right)  $. We will call $\left(  M,\Delta\right)  $ a
\textit{compact} quantum group if we can take $\varphi=\psi$ as a state, which
we will call the \textit{Haar state}. Note that the Haar state is faithful and normal.

The dual $\left(  \hat{M},\hat{\Delta}\right)  $ of $\left(
M,\Delta\right) $ is again a locally compact quantum group and is
defined as follows (also see \cite[Definition 3.1]{vD}), where we
assume $M$ is in standard form with respect to the Hilbert space
$H$: Denote by $W\in M\otimes B(H)$ the so-called multiplicative
unitary of $\left( M,\Delta\right)  $ with respect to the GNS
construction on $H$ obtained from some $\varphi$ as above; see
\cite[Theorem 1.2]{KV4}. Let $\hat{\lambda }:M_{\ast}\rightarrow
B(H):\theta\mapsto(\theta\otimes\iota)(W)$ with $\iota$ the identity
on $B(H)$. Then $\hat{M}$ is defined to be the $\sigma$-weak closure
of $\left\{ \hat{\lambda}(\theta):\theta\in M_{\ast}\right\}  $ and
$\hat{\Delta}$ is defined by $\hat{\Delta}(a):=\Sigma
W(a\otimes1)W^{\ast }\Sigma$ where $\Sigma:H\otimes H\rightarrow
H\otimes H$ is the ``flip map'' and $1\in M$ is the identity
operator on $H$. The symbol $\iota$ will always denote the identity
map on some von Neumann algebra which will be clear from context.

Next we give the basic definitions and results which we use to build our
dynamical systems.

\begin{definition}
An \textit{automorphism} of $\left(M,\Delta\right)$ is a
$\ast$-automorphism $\alpha:M\rightarrow M$ such that
$\Delta\circ\alpha=(\alpha\otimes\alpha)\circ\Delta$.
\end{definition}

\begin{proposition}
Let $\alpha$ be an automorphism of a compact quantum group
$\left(A,\Delta\right)$ with Haar state $\mu$. Then
$\mu\circ\alpha=\mu$.
\end{proposition}

\begin{proof}
From the strong form of left invariance
\cite[Proposition 3.1]{KV4} we have
$\mu\circ\alpha(a)1=\alpha^{-1}[\mu(\alpha(a))1]=\alpha
^{-1}[(\iota\otimes\mu)\circ\Delta\circ\alpha(a)]=\alpha^{-1}\circ
(\iota\otimes\mu)\circ(\alpha\otimes\alpha)\circ\Delta(a)=[\iota\otimes
(\mu\circ\alpha)]\circ\Delta(a)$ but this says that $\mu\circ\alpha$
is also left invariant, \ hence by uniqueness of left invariant
states (which also explains the terminology \textit{the} Haar state)
we have $\mu\circ\alpha=\mu $.
\end{proof}

Hence $\mathbf{A=}\left(A,\mu,\alpha\right)$ is a dynamical system
with $G=\mathbb{Z}$ by simply setting $\alpha_{n}:=\alpha^{n}$ for
$n\in\mathbb{Z}$. Let us now look at the specific case that will
interest us throughout the rest of this section, and also fix the
notation that we will use:

Let $\Gamma$ be any group and assign to it the discrete topology and counting
measure. We set
\begin{equation}
\left(A,\Delta\right):=\left(\widehat{L^{\infty}(\Gamma)},\hat{\Delta}_{\Gamma}\right)
\tag{3.1}
\end{equation}
where $[\Delta_{\Gamma}(f)](g,h):=f(gh)$ for all $f\in
L^{\infty}(\Gamma)$ and $g,h\in\Gamma$, and where of course we view
the elements of $L^{\infty} (\Gamma)$ as linear operators on
$H:=L^{2}(\Gamma)$ by multiplication. In this situation we in fact
have that $A$ is generated by $\left\{  \lambda
(g):g\in\Gamma\right\}  $ where we write
$\lambda(g)\equiv\lambda(\theta_{g})$ with
$\lambda:L^{\infty}(\Gamma)_{\ast}\rightarrow B(H)$ defined in the
same way as $\hat{\lambda}$ above, and $\theta_{g}(f):=f(g)$, which
translates into $\lambda:\Gamma\rightarrow B(H)$ being a unitary
representation of $\Gamma$ with $[\lambda(g)f](h)=f(g^{-1}h)$ for
all $f\in H $ and $g,h\in\Gamma$, and having the property
$\Delta(\lambda(g))=\lambda(g)\otimes\lambda(g)$. In this case we
also have that the Haar state $\mu$ is tracial, i.e.
$\mu(ab)=\mu(ba)$ for all $a,b\in M$, however this doesn't play a
direct role in our further work. Furthermore, let
\[
T:\Gamma\rightarrow\Gamma
\]
be any automorphism of the group $\Gamma$. From $T$ we now obtain an
automorphism of $\left(  A,\Delta\right)  $ as follows: Define a
unitary operator $U:H\rightarrow H$ by $Uf:=f\circ T$. Since
$\left[U^{\ast} \lambda(g)Uf\right] (h)=\left[ \lambda(g)(f\circ
T)\right] \left( T^{-1}(h)\right) =\left( f\circ T\right) \left(
g^{-1}T^{-1}(h)\right) =f\left( T(g)^{-1}h\right) =\left[
\lambda(T(g))f\right] (h)$, we have
\[
U^{\ast}\lambda(g)U=\lambda(T(g))
\]
which in particular means that the set of generators of $A$ is
invariant under $U^{\ast}(\cdot)U$ and hence $A$ itself as well. So
we have a well-defined mapping
\[
\alpha:A\rightarrow A:a\mapsto U^{\ast}aU
\]
which we call the \textit{dual} of $T$. It remains to show that $\alpha$ is an
automorphism of $\left(  A,\Delta\right)  $. Note that $\Delta\circ
\alpha(\lambda(g))=\Delta(\lambda(T(g)))=\lambda(T(g))\otimes\lambda
(T(g))=(\alpha\otimes\alpha)(\lambda(g)\otimes\lambda(g))=(\alpha\otimes
\alpha)\circ\Delta(\lambda(g))$, and by linearity and $\sigma$-weak continuity
this extends to all of $A$, that is to say $\Delta\circ\alpha=(\alpha
\otimes\alpha)\circ\Delta$ as required.

We will refer to the dynamical system
$\mathbf{A=}\left(A,\mu,\alpha\right)$ as the \textit{dual system}
of $\left(\Gamma,T\right)$, and this notation will be fixed
throughout the rest of this section. Our eventual goal in this
section is a refinement of Theorems 2.1 (one direction) and 2.7 for
dual systems, however we first develop some general theory regarding
dual systems.

As we show next, every automorphism of $\left(A,\Delta\right)$ is
the dual of some automorphism of $\Gamma$, hence assuming the
automorphism $T$ of $\Gamma$ to be given places no restriction on
the dynamics obtained as automorphisms of $\left(A,\Delta\right)$.

We will use the following additional notation: By $\delta_{g}$ with
$g\in\Gamma$, we denote the element of $H$ defined by
$\delta_{g}(g)=1$ and $\delta_{g}(h)=0$ for $g\neq h\in\Gamma$. In
particular we set $\Omega :=\delta_{1}$ where $1$ here denotes the
identity of $\Gamma$. Then $\Omega$ is cyclic and separating for
$A$, and $\mu(a)=\left\langle \Omega,a\Omega \right\rangle $ so
$\left(  H,\iota_{A},\Omega\right)  $ is the cyclic representation
of $\left(  A,\mu\right)  $ obtained in the GNS construction. Also
note that $\lambda(g)\Omega=\delta_{g}$. We will use the notation
$\chi_{g}:=\delta_{g}$ when we want to view this function as an
element of $L^{\infty}(\Gamma)$ rather than $H=L^{2}(\Gamma)$; this
makes some the arguments slightly easier to read. Using the notation
$\gamma:A\rightarrow H:a\mapsto a\Omega$, the multiplicative unitary
$W$ of $\left(  A,\Delta \right)  $ has the following defining
property (see \cite[Theorem 1.2]{KV4}):
$W^{\ast}\left[\gamma(a)\otimes\gamma(b)\right]=(\gamma\otimes
\gamma)\left[\Delta(b)(a\otimes1)\right]$.

\begin{theorem}
Every automorphism $\alpha$ of $\left(A,\Delta\right)$ in (3.1) is
the dual of some automorphism $T$ of the discrete group $\Gamma$.
\end{theorem}

\begin{proof}
Using the notation above, we define a unitary
operator
$U:H\rightarrow H$ by $U^{\ast}a\Omega:=\alpha(a)\Omega$. We first show that%
\begin{equation}
(U\otimes U)W=W(U\otimes U) \tag{3.2}
\end{equation}
to enable us to define an automorphism of $\left(
\hat{A},\hat{\Delta }\right)  $. Using the defining property of $W$
we have
\begin{align*}
(U^{\ast}\otimes U^{\ast})W^{\ast}(\delta_{g}\otimes\delta_{h})  &  =(U^{\ast
}\otimes U^{\ast})(\gamma\otimes\gamma)\left[  \Delta(\lambda(h))(\lambda
(g)\otimes1)\right] \\
&  =(U^{\ast}\otimes U^{\ast})(\gamma\otimes\gamma)\left[  \lambda
(hg)\otimes\lambda(h)\right] \\
&  =\left[  \gamma\circ\alpha\circ\lambda(hg)\right]  \otimes\left[
\gamma\circ\alpha\circ\lambda(h)\right] \\
&  =(\gamma\otimes\gamma)\left(  \left[  (\alpha\otimes\alpha)\circ
\Delta(\lambda(h))\right]  \left[  (\alpha\circ\lambda(g))\otimes1\right]
\right) \\
&  =(\gamma\otimes\gamma)\left(  \left[  \Delta(\alpha(\lambda(h)))\right]
\left[  \alpha(\lambda(g))\otimes1\right]  \right) \\
&  =W^{\ast}(U^{\ast}\otimes U^{\ast})(\delta_{g}\otimes\delta_{h})
\end{align*}
which proves (3.2). For any $\theta\in A_{\ast}$ it follows from the
definition of $\hat{\lambda}$ and from (3.2) that
\begin{align*}
U\hat{\lambda}(\theta)U^{\ast}  &  =(\theta\otimes\iota)\left[  (1\otimes
U)W(1\otimes U^{\ast})\right] \\
&  =(\theta\otimes\iota)\left[  (U^{\ast}\otimes1)W(U\otimes1)\right] \\
&  =\left[  (\theta\circ\alpha)\otimes\iota\right]  (W)\\
&  =\hat{\lambda}(\theta\circ\alpha)
\end{align*}
and therefore $U\hat{A}U^{\ast}=\hat{A}$ so
\[
\hat{\alpha}:\hat{A}\rightarrow\hat{A}:a\mapsto UaU^{\ast}
\]
is a well-defined $\ast$-automorphism of $\hat{A}$. Next observe
from the definition of $\hat{\Delta}$ and again using (3.2) that
\begin{align*}
(\hat{\alpha}\otimes\hat{\alpha})\circ\hat{\Delta}(a)  &  =(U\otimes U)\Sigma
W(a\otimes1)W^{\ast}\Sigma(U^{\ast}\otimes U^{\ast})\\
&  =\Sigma W(\hat{\alpha}(a)\otimes1)W^{\ast}\Sigma\\
&  =\hat{\Delta}\circ\hat{\alpha}(a)
\end{align*}
in other words $\hat{\alpha}$ is an automorphism of
$\left(\hat{A},\hat {\Delta}\right)$. However, by Pontryagin duality
(see for example \cite[p. 75]{KV4}) $\left(
\hat{A},\hat{\Delta}\right) =\left(  L^{\infty}
(\Gamma),\Delta_{\Gamma}\right)  $. With somewhat tedious but fairly
elementary arguments one can then show that $\hat{\alpha}(f)=f\circ
T$ for all $f\in L^{\infty}(\Gamma)$ for some automorphism $T$ of
the group $\Gamma$. Lastly we show that $\alpha$ is the dual of $T$
in the sense defined earlier. Define $U_{T}:H\rightarrow H:f\mapsto
f\circ T$, then $U_{T}^{\ast}
fU_{T}\delta_{g}=f(T^{-1}(g))\delta_{g}=U^{\ast}fU\delta_{g}$ for
all $f\in L^{\infty}(\Gamma)$, so $U_{T}U^{\ast}f=fU_{T}U^{\ast}$,
in particular for $f=\chi_{h}$. Hence for all $g\neq h$ in $\Gamma$
we have $\chi_{h} U_{T}U^{\ast}\delta_{g}=0$ and so
$U^{\ast}\delta_{g}=k(g)U_{T}^{\ast} \delta_{g}=k(g)\delta_{T(g)}$
for some complex number $k(g)$ of modulus $1$. But then
$\alpha(\lambda(g))\Omega=U^{\ast}\delta_{g}=k(g)\lambda(T(g))\Omega$
and therefore
$k(g)\lambda(T(g))\otimes\lambda(T(g))=\Delta(k(g)\lambda
(T(g))=(\alpha\otimes\alpha)\circ\Delta(\lambda(g))=[k(g)]^{2}\lambda
(T(g))\otimes\lambda(T(g))$ since $\Omega$ is separating for $A$, so
$k(g)=1$ which means that $U_{T}=U$ as required.
\end{proof}

Having set up the framework, we can now start doing ergodic theory. We first
discuss the theorem of Halmos in the current setting. Recall that $\mathbf{A}
$ is \textit{strongly mixing} if
\[
\lim_{n\rightarrow\infty}\mu(\alpha^{n}(a)b)=\mu(a)\mu(b)
\]
for all $a,b\in A$. We will say that $g\in\Gamma$ has a finite orbit
under $T$ if the orbit
$T^{\mathbb{N}}(g):=\left\{T^{n}(g):n\in\mathbb{N}\right\}$ is a
finite set, where $\mathbb{N=}\left\{1,2,3,...\right\}$.

The following result is fairly standard, but often expressed in terms of group
C*-algebras, or in the language of group von Neumann algebras (see
\cite[2.12]{A} for an example of this type of result). For completeness,
and since we use this result later, we
include a proof based on that of the abelian case in for example \cite[Section
2.5]{P}.

\begin{theorem}
If the dual system $\mathbf{A}$ is ergodic, then the only element of
$\Gamma$ with a finite orbit under $T$ is its identity $1$.
Conversely, if $1$ is the only element of $\Gamma$ with finite orbit
under $T$, then $\mathbf{A}$ is strongly mixing. It follows that
$\mathbf{A}$ is strongly mixing if and only if it is ergodic.
\end{theorem}

\begin{proof}
Suppose $g\in\Gamma\backslash\{1\}$ has a finite orbit
under $T$. Then there is a smallest $n\in\mathbb{N}$ such that
$T^{n}g=g$. Hence this is the smallest $n$ in $\mathbb{N}$ for which
$(U^{\ast})^{n}\delta _{g}=\delta_{g}$. Set
\[
x:=\delta_{g}+U^{\ast}\delta_{g}+(U^{\ast})^{2}\delta_{g}+...+(U^{\ast})^{n-1}\delta_{g}
\]
then $U^{\ast}x=x$. It is also easily seen that $U^{\ast}\Omega=\Omega$, but
as we now show, $x\notin\mathbb{C}\Omega$. Since $x=\delta_{g}+\delta
_{T(g)}+...+\delta_{T^{n-1}(g)}$, while $g\neq1$ and hence $T^{j}(g)\neq1$, we
have $x(1)=0\neq1=\Omega(1)$. At the same time $x(g)=1$ so $x\neq0$, so
$x\notin\mathbb{C}\Omega$. This means the fixed point space of $U^{\ast}$ has
dimension larger than $1$, and therefore $\mathbf{A}$ is not ergodic.

Conversely, suppose $1$ is the only element of $\Gamma$ with finite orbit
under $T$. Consider $g,h\in G$. If $g=h=1$, then it is easily seen that
$\lim_{n\rightarrow\infty}\left\langle (U^{\ast})^{n}\delta_{g},\delta
_{h}\right\rangle =1=\left\langle \delta_{g},\Omega\right\rangle \left\langle
\Omega,\delta_{h}\right\rangle $. Otherwise, if at least one of $g$ or $h$ is
not $1$, then from our supposition, $T^{n}(h)\neq g$ for $n$ large enough, and
therefore it is again easily seen that $\lim_{n\rightarrow\infty}\left\langle
(U^{\ast})^{n}\delta_{g},\delta_{h}\right\rangle =0=\left\langle \delta
_{g},\Omega\right\rangle \left\langle \Omega,\delta_{h}\right\rangle $. From
this we deduce that
\[
\lim_{n\rightarrow\infty}\left\langle (U^{\ast})^{n}x,y\right\rangle
=\left\langle x,\Omega\right\rangle \left\langle \Omega,y\right\rangle
\]
for all $x,y\in H$, but this means that $\mathbf{A}$ is strongly
mixing.
\end{proof}

Weak mixing is an intermediate condition between ergodicity and strong mixing
and is therefore also equivalent to these two conditions. Another simple
corollary of Theorem 3.4 (which can also be seen directly) is that if
$1<\left|  \Gamma\right|  <\infty$, then $\mathbf{A}$ cannot be ergodic.

We now move on to subsystems and compactness. In our current
situation, if $\mathbf{A}$ is not weakly mixing then it is not
ergodic, and so one can obtain a nontrivial compact subsystem by
considering the fixed point algebra of $\alpha$. But a result purely
in terms of dual systems would be preferable, and from the point of
view of weak mixing we want a result in terms of a compact subsystem
that need not be an identity system. Hence we consider the
following:

Let $E:=\left\{  g\in\Gamma:T^{\mathbb{N}}(g)\text{ is finite}\right\}  $ and
let $F$ denote the von Neumann algebra generated by $\left\{  \lambda(g):g\in
E\right\}  $. Then

\begin{theorem}
The system $\mathbf{F}=\left(F,\kappa ,\varphi\right)
:=\left(F,\mu|_{F},\alpha|_{F}\right)$ is isomorphic to the dual
system of $\left( E,T|_{E}\right)$ and it is a compact subsystem of
$\mathbf{A}$. Furthermore, if $\mathbf{F}$ is trivial then
$\mathbf{A}$ is ergodic.
\end{theorem}

\begin{proof}
One easily sees that $T|_{E}$ is an automorphism of the subgroup $E$
of $\Gamma$, hence $\alpha(F)=F$. So $\mathbf{F}$ is indeed a
subsystem of $\mathbf{A}$. It is also readily seen that if $K$ is
the closure of $F\Omega$ in $H$, then $\pi:F\rightarrow
B(K):a\mapsto a|_{K}$ is well-defined and $\left(
K,\pi,\Omega\right)  $ is the cyclic representation of $\left(
F,\kappa\right)  $ obtained in the GNS construction. Also note that
$K$ is the closure of $D:=$ span$\left\{  \delta_{g}:g\in E\right\}
$.

Note that $\pi$ is injective since $\Omega$ is separating for $F$,
and then one can verify that $\pi(F)$ is generated by
$\lambda_{E}:E\rightarrow B(K)$ where $\left[ \lambda_{E}(g)f\right]
(h):=f(g^{-1}h)$ for $f\in K=L^{2}(E)$ in terms of the counting
measure on $E$. So $\pi(F)=\widehat{L^{\infty}(E)}$. It is readily
verified that $\pi$ is an isomorphism (as defined in Section 1) of
the dynamical system $\mathbf{F}$ and the dual system of
$\left(E,T|_{E}\right)$.

Consider any $v=\sum_{j=1}^{r}c_{j}\delta_{g_{j}}$ in $D$ and let
$n_{g}=\left|  T^{\mathbb{N}}(g)\right|  $ denote the length of
$g$'s orbit, i.e. it is the smallest element of $\mathbb{N}$ such
that $T^{n_{g}}(g)=g$. Then $(U^{\ast})^{n_{g_{1}}...n_{g_{r}}}v=v$,
in other words $v$ also has a finite orbit. For arbitrary $x\in K$
and $\varepsilon>0$ there will be a $v\in D$ with $\left\|
x-v\right\|  <\varepsilon$ and therefore
$\left\|(U^{\ast})^{n}x-(U^{\ast})^{n}v\right\| <\varepsilon$ for
all $n\in\mathbb{N}$, but since $(U^{\ast})^{\mathbb{N}}v$ is
finite, $(U^{\ast})^{\mathbb{N}}x $ is totally bounded. We conclude
that $\mathbf{F}$ is compact.

When $\mathbf{A}$ is not ergodic, $E\neq\{1\}$ by Theorem 3.4, and
therefore $K\neq\mathbb{C}\Omega$ so $F\neq\mathbb{C}1$ by the
definition of $K$. In other words, $\mathbf{F}$ is nontrivial.
\end{proof}

We will refer to $\mathbf{F}$ defined above as the \textit{finite
orbit subsystem} of $\mathbf{A}$.

Before we proceed with subsystems and joinings, we give the
following analogue of Theorem 3.4 as an application of (part of)
Theorem 3.5:

\begin{theorem}
The dual system $\mathbf{A}$ is compact if and only if all the
orbits in $\left(\Gamma,T\right)$ are finite.
\end{theorem}

\begin{proof}
Suppose $\mathbf{A}$ is compact. Then in particular for any
$g\in\Gamma$ the orbit
$\delta_{T^{\mathbb{N}}(g)}:=\left\{\delta_{T^{n}
(g)}:n\in\mathbb{N}\right\} =(U^{\ast})^{\mathbb{N}}\delta_{g}$ is
totally bounded in $H$. Hence there is a finite set $N\subset H$
such that for every $\delta_{T^{n}(g)}$ there is an $x\in N$ with
$\left\| \delta_{T^{n}(g)}-x\right\| <1/\sqrt{2}$. However, for any
pair $h\neq j$ in $\Gamma$ we have
$\left\|\delta_{h}-\delta_{j}\right\| =\sqrt{2}$, so for any $x\in
N$ the ball $\left\{y\in H:\left\| y-x\right\| <1/\sqrt{2}\right\}$
contains at most one point in the orbit
$\delta_{T^{\mathbb{N}}(g)}$. Since $N$ is finite it follows that
$T^{\mathbb{N}}(g)$ is finite.

Conversely, assume that all the orbits $T^{\mathbb{N}}(g)$ in
$\Gamma$ are finite, i.e. $E=\Gamma$, so $F=A$ and therefore
$\mathbf{A}$ is compact by Theorem 3.5.
\end{proof}

Using Theorems 3.4 and 3.6, one can now in a standard way easily construct
concrete examples of dynamical systems which are either ergodic or compact or
neither. For example if $\Gamma$ is a free group generated by an alphabet $S$
and $T:S\rightarrow S$ is a bijection which we extend to a automorphism
$T:\Gamma\rightarrow\Gamma$ then we get a dual system which is ergodic or
compact or neither depending on whether the orbits of $T$ on $S$ are all
infinite or all finite or neither. Another example is to consider the group
$\Gamma$ of finite permutations of a possibly infinite set $S$ with
automorphisms given by $g\mapsto h^{-1}gh$ where $h$ is an element of the
group of all bijections of $S$, in which case one can again obtain ergodicity
or compactness or neither by choosing $h$ appropriately.

Next we mention a simple converse for Theorem 3.5:

\begin{proposition} If $\mathbf{A}$ is ergodic,
then it has no nontrivial compact subsystems.
\end{proposition}

\begin{proof}
Note that $\mathbf{A}$ is weakly mixing by Theorem 3.4 and hence
does not have a nontrivial compact subsystem by \cite[Theorem
6.8]{BDS}.
\end{proof}

We now reach our final goal for this section, namely to make a
connection with joinings.

\begin{theorem}
If the dual system $\mathbf{A}$ is disjoint from all compact dual
systems, then it is ergodic.
\end{theorem}

\begin{proof}
Define a group $\tilde{\Gamma}$ that consists of the same elements
as $\Gamma$ but with the product $g\cdot h:=hg$. Then $T$ is an
automorphism of $\tilde{\Gamma}$. Let $\mathbf{B}$ be the dual
system of $\left(  \tilde{\Gamma},T\right)  $. It is easily verified
that $B\subset A^{\prime}$. Let $\mathbf{F}$ be the finite orbit
subsystem of $\mathbf{B}$, so in particular $\mathbf{F}$ is compact
and isomorphic to a dual system by Theorem 3.5. Then we see that
$\omega:A\odot F\rightarrow\mathbb{C} :t\mapsto\left\langle
\Omega,\delta(t)\Omega\right\rangle $ is a joining of $\mathbf{A}$
and $\mathbf{F}$ where $\delta:A\odot F\rightarrow B(H)$ is defined
through $\delta(a\otimes b)=ab$. (This is again the ``diagonal
measure'' idea.) Note that as in \cite[Lemma 3.5]{D}, $\omega$ is
trivial (i.e. equal to $\mu\odot\nu$) if and only if $\mathbf{F}$ is
trivial. But since $\left(\tilde{\Gamma},T\right)$ has the same
orbits as $\left(\Gamma,T\right)$, we know from Theorem 3.4 that
$\mathbf{B}$ is ergodic if and only if $\mathbf{A}$ is. Hence, if we
assume that $\mathbf{A}$ is not ergodic, then $\mathbf{F}$ is
nontrivial by Theorem 3.5.
\end{proof}

Proposition 3.7 suggests that the converse of Theorem 3.8 might be true,
however I don't have a proof or a counter example.

\section{Ergodic joinings}

In this section we briefly motivate and study ergodic joinings. We
begin by noting that for our systems $\mathbf{A}$ and $\mathbf{B}$
from Section 1, every state on the (unital) $\ast$-algebra $A\odot
B$ can in fact be extended to a state on the maximal C*-algebraic
tensor product $A\otimes_{m}B$. This is a consequence of the
following proposition pointed out to me by the referee, which is
certainly known, but for which I have no reference.

\begin{proposition}
Let $A$ and $B$ be unital C*-algebras, and $\omega$ any state on
their algebraic tensor product $A\odot B$. Then $\omega$ is bounded
with respect to the maximal C*-norm on $A\odot B$.
\end{proposition}

\begin{proof}
In this proof the notation $s\leq t$ for $s,t\in A\odot B$ means
that $t-s$ is a finite sum of terms of the form $u^{\ast}u$ with
$u\in A\odot B$. The proof has two steps.

Firstly we assume the so-called \textit{Axiom} $A_{1}$ of F. Combes
\cite[p. 38]{C}, namely that for every $t\in A\odot B$ there exists
a scalar $\lambda_{t}\geq0$ such that
\[
s^{\ast}t^{\ast}ts\leq\lambda_{t}s^{\ast}s
\]
for all $s\in A\odot B$. We now show that from this assumption it
follows that $\omega$ is bounded with respect to the maximal C*-norm
$\left\| \cdot\right\| _{m}$ on $A\odot B$. Let
$\left(H_{\omega},\pi_{\omega},\Omega_{\omega}\right)$ be the cyclic
representation of $\left(A\odot B,\omega\right)$ obtained from the
GNS construction, so we have a linear $\gamma_{\omega}:A\odot
B\rightarrow H_{\omega}$ such that $\gamma_\omega(A\odot B)$ is
dense in $H_\omega$,
$\langle\gamma_\omega(s),\gamma_\omega(t)\rangle = \omega(s^*t)$,
and
\[
\pi_{\omega}(t)\gamma_{\omega}(s):=\gamma_{\omega}(ts)
\]
for all $s,t\in A\odot B$. Then
\[
\left\|
\pi_{\omega}(t)\gamma_{\omega}(s)\right\|^{2}=\omega\left((ts)^{\ast}ts\right)
\leq\lambda_{t}\omega(s^{\ast}s)=\lambda_{t}\left\|
\gamma_{\omega}(s)\right\| ^{2}
\]
for all $s,t\in A\odot B$, and since $\gamma_{\omega}(A\odot B)$ is
dense in $H_{\omega}$, it follows that $\pi_{\omega}(t)$ can be
extended to a bounded linear operator on $H_{\omega}$. As for the
cyclic representation of a state on a C*-algebra,
$\pi_{\omega}:A\odot B\rightarrow B(H_{\omega})$ is a
$\ast$-homomorphism, for example $\left\langle
\pi_{\omega}(t)^{\ast}\gamma
_{\omega}(u),\gamma_{\omega}(s)\right\rangle =\left\langle
\gamma_{\omega }(u),\gamma_{\omega}(ts)\right\rangle
=\omega(u^{\ast}ts)=\overline
{\omega(s^{\ast}t^{\ast}u)}=\overline{\left\langle \gamma_{\omega}
(s),\pi_{\omega}(t^{\ast})\gamma_{\omega}(u)\right\rangle }$ from
which $\pi_{\omega}(t^{\ast})=\pi_{\omega}(t)^{\ast}$ follows. This
implies that $t\mapsto\left\|  \pi_{\omega}(t)\right\|  $ is a
C*-seminorm on $A\odot B$, and therefore $\left\|
\pi_{\omega}(t)\right\|  \leq\left\|  t\right\|  _{m}$ for $t\in
A\odot B$; see \cite[p. 193]{M} for example. But using the cyclic
representation we then have
\[
\left|  \omega(t)\right|  =\left|  \left\langle
\Omega_{\omega},\pi_{\omega }(t)\Omega_{\omega}\right\rangle \right|
\leq\left\| \pi_{\omega }(t)\right\| \leq\left\| t\right\| _{m}
\]
and hence $\omega$ is bounded with respect to the maximal C*-norm as
required.

Secondly we show that Combes' axiom is indeed satisfied, and this
will complete the proof. Note that for $0\leq a\in A$ and $0\leq
b\in B$ we have $a\otimes b=\left(  a^{1/2}\otimes b^{1/2}\right)
^{\ast}\left( a^{1/2}\otimes b^{1/2}\right)  \geq0$. For $a_{2}\geq
a_{1}$ in $A$ and $b_{2}\geq b_{1}$ in $B$ it follows that
$a_{2}\otimes b_{2}-a_{1}\otimes
b_{1}=(a_{2}-a_{1})\otimes(b_{2}-b_{1})+a_{1}\otimes(b_{2}-b_{1})+(a_{2}-a_{1})\otimes
b_{1}\geq0$, hence
\[
a_{2}\otimes b_{2}\geq a_{1}\otimes b_{1}.
\]
For an arbitrary $t=\sum_{k=1}^{n}a_{k}\otimes b_{k}\in A\odot B$ it
follows from \cite[Inequality 8.5]{NSZ}, the inequality above, and
the fact that for $c\geq0$ in a unital C*-algebra one has
$c\leq\left\|  c\right\|  $, that
\[
t^{\ast}t\leq
n\sum_{k=1}^{n}(a_{k}^{\ast}a_{k})\otimes(b_{k}^{\ast}b_{k})\leq\left(n\sum_{k=1}^{n}\left\|
a_{k}\right\| ^{2}\left\| b_{k}\right\| ^{2}\right)
1_{A}\otimes1_{B}
\]
hence Combes' axiom holds with $\lambda_{t}=n\sum_{k=1}^{n}\left\|
a_{k}\right\| ^{2}\left\|  b_{k}\right\| ^{2}$.
\end{proof}

In particular for any systems $\mathbf{A}$ and $\mathbf{B}$ it
follows that in effect every element $\omega$ of
$J(\mathbf{A},\mathbf{B})$ is a state on $A\otimes_{m}B$, and by
continuity we have $\omega\circ\left(  \alpha
_{g}\otimes_{m}\beta_{g}\right)  =\omega$ for all $g\in G$. In the
rest of this section we work in terms of this setting.

\begin{proposition}
The set $J\left( \mathbf{A},\mathbf{B}\right)$ is weakly* compact,
and it is the closed convex hull of its extreme points. In
particular this set of extreme points, which we will denote by
$J_{e}\left(\mathbf{A},\mathbf{B}\right)$, is not empty.
\end{proposition}

\begin{proof}
Let $S$ be the set of states on $A\otimes_{m}B$. Since $S$ is
weakly* compact, and it is readily verified that
$J\left(\mathbf{A},\mathbf{B}\right)$ is weakly* closed in $S$, it
follows that $J\left(\mathbf{A},\mathbf{B}\right)$ is weakly*
compact. It is easy to see $J\left(\mathbf{A},\mathbf{B}\right)$ is
convex. Since $\mu\otimes _{m}\nu\in
J\left(\mathbf{A},\mathbf{B}\right)  $, it follows from the
Krein-Milman theorem that $J_{e}\left(\mathbf{A},\mathbf{B}\right)$
is not empty and that $J\left(\mathbf{A},\mathbf{B}\right)$ is the
closed convex hull of $J_{e}\left(\mathbf{A},\mathbf{B}\right)$.
\end{proof}

\begin{definition}
A \textit{C*-dynamical system} $\left(C,\tau\right)$ consists of a
unital C*-algebra $C$ and a representation
$\tau:G\rightarrow$Aut$(C):g\mapsto\tau_{g}$ of a group $G$. Let
$E_{\tau}$ denote the extreme points of the set of $\tau$-invariant
states on $C$.
\end{definition}

This set of extreme points connects to ergodicity (in the sense that
we have been using the term) in the following way: By \cite[Theorem
4.3.20]{BR} a W*-dynamical system is ergodic if and only if the
unitary representation of its dynamics on its GNS Hilbert space has
a one-dimensional fixed point space. On the other hand, according to
\cite[Theorem 4.3.17]{BR}, if the triple $\left(C,\rho,\tau\right)$
is $G$-abelian (see \cite[Definition 4.3.6]{BR}) for some
$\tau$-invariant state $\rho$ on $C$, then $\rho\in E_{\tau}$ if and
only if the unitary representation of $\tau$ on the GNS Hilbert
space of $\left(C,\rho\right)$ has a one-dimensional fixed point
space. So in this case we can say that every element $\rho$ of
$E_{\tau}$ gives us an ergodic C*-dynamical system of the form
$\left(C,\rho,\tau\right)$. The latter will appear again in Section
5, but without the assumption that it is $G$-abelian.

\begin{proposition}
If $\mathbf{A}$ and $\mathbf{B}$ are ergodic, then
$J_{e}\left(\mathbf{A},\mathbf{B}\right) \subset
E_{\alpha\otimes_{m}\beta}$ where
$\left(\alpha\otimes_{m}\beta\right)_{g}:=\alpha_{g}\otimes_{m}\beta_{g}$.
\end{proposition}

\begin{proof}
Since $\mathbf{A}$ and $\mathbf{B}$ are ergodic, we have $\mu\in
E_{\alpha}$ and $\nu\in E_{\beta}$; see for example \cite[Theorem
4.3.17]{BR}. Now consider any $\omega\in
J_{e}\left(\mathbf{A},\mathbf{B}\right)$ and write
$\omega=r\omega_{1}+(1-r)\omega_{2}$ where $\omega_{1}$ and
$\omega_{2}$ are states invariant under $\alpha\otimes _{m}\beta$,
and $0<r<1$. Then $\mu=\omega\left( \cdot\otimes1_{B}\right)
=r\omega_{1}\left( \cdot\otimes1_{B}\right) +(1-r)\omega_{2}\left(
\cdot\otimes1_{B}\right)  $, but $\mu\in E_{\alpha}$, hence
$\mu=\omega _{j}\left(  \cdot\otimes1_{B}\right) $ and likewise
$\nu=\omega_{j}\left(1_{A}\otimes\cdot\right)$. Thus $\omega_{j}\in
J\left(\mathbf{A},\mathbf{B}\right)$, but $\omega$ is extremal in
the latter set, therefore $\omega=\omega_{j}$. This shows that
$\omega\in E_{\alpha\otimes_{m}\beta}$.
\end{proof}

This proposition motivates the term \textit{ergodic joining} (of
$\mathbf{A}$ and $\mathbf{B}$) for each element of
$J_{e}\left(\mathbf{A} ,\mathbf{B}\right)$ when $\mathbf{A}$ and
$\mathbf{B}$ are both ergodic.

We end this section with another illustration of how joinings can in principle
be used, by proving a Halmos-von Neumann type theorem for W*-dynamical systems
in terms of Hilbert space. Unfortunately we require a form of asymptotic
abelianness defined as follows:

\begin{definition}
Consider a C*-dynamical system $\left(C,\tau\right)$ whose group $G$
is countable, discrete and amenable. Let $\left(\Lambda_{n}\right)$
be any F\o lner sequence in $G$. If
\begin{equation}
\lim_{n\rightarrow\infty}\frac{1}{\left|  \Lambda_{n}\right|  }\sum
_{g\in\Lambda_{n}}\left\|  \left[  a,\tau_{g}(b)\right]  \right\|
=0 \tag{4.1}
\end{equation}
for all $a,b\in C$ where $\left[\cdot,\cdot\right]$ is the
commutator, then we say $\left(C,\tau\right)$ is
$\left(\Lambda_{n}\right)$\textit{-asymptotically abelian}.
\end{definition}

This type of asymptotic abelianness was also used in \cite{NSZ} for the case
$G=\mathbb{Z}$. We will not in fact need any properties of F\o lner sequences;
we will only use (4.1), for example it does not matter if $\left(  \Lambda
_{n}\right)  $ is a right or left F\o lner sequence.

\begin{proposition}
Let $\mathbf{A}$ and $\mathbf{B}$ be ergodic,
$\left(\Lambda_{n}\right)$-asymptotically abelian and have the same
point spectrum, i.e. $\sigma_{\mathbf{A}}=\sigma_{\mathbf{B}}$. Then
the unitary representations $U$ and $V$ of $\alpha$ and $\beta$
respectively (as in Definition 2.5) can be done on Hilbert subspaces
of some Hilbert space, such that the eigenvectors of $U$ and $V$
span the same Hilbert subspace, say $H_{0}$, and such that
$U_{g}x=V_{g}x$ for all $x\in H_{0}$ and $g\in G$.
\end{proposition}

\begin{proof}
We follow the basic plan due to Lema\'{n}czyk \cite{L} (also see
\cite[Theorem 7.1]{G}) for the measure theoretic case. By
Proposition 4.2 there exists an $\omega\in
J_{e}\left(\mathbf{A},\mathbf{B}\right)$. Note furthermore that
$\left(A\otimes_{m}B,\alpha\otimes_{m}\beta\right)$ is
$\left(\Lambda_{n}\right)$-asymptotically abelian, and hence it is
easy to see that the pair $\left(  A\otimes_{m}B,\omega\right)$ is
$G$-abelian (see \cite[Definition 4.3.6]{BR}). Now consider the
``combined'' GNS construction for
$\left(A\otimes_{m}B,\omega\right)$, $\left( A,\mu\right)$ and
$\left(B,\nu\right)$ as given by \cite[Construction 2.3]{D}, namely
$\left( H_{\omega},\gamma_{\omega}\right)$,
$\left(H_{\mu},\gamma_{\mu}\right)$ and
$\left(H_{\nu},\gamma_{\nu}\right)$, and the corresponding unitary
representations $W$, $U$ and $V$ of $\alpha\otimes_{m}\beta$,
$\alpha$ and $\beta$ respectively. From
$\left(H_{\omega},\gamma_{\omega}\right)$ and
$\left(H_{\mu},\gamma_{\mu}\right) $ we of course also obtain the
respective cyclic representations with common cyclic vector:
$\left(H_{\omega},\pi_{\omega},\Omega_{\omega}\right)$ and
$\left(H_{\mu},\pi_{\mu},\Omega_{\omega}\right)$.

Take any $\chi\in\sigma_{\mathbf{A}}=\sigma_{\mathbf{B}}$ then by
\cite[Theorem 2.5]{S} the corresponding eigenvectors of $U$ and $V$ are of the
form $\gamma_{\mu}(a)$ and $\gamma_{\nu}(b)$ for some $a\in A$ and $b\in B$,
and furthermore $\alpha_{g}(a)=\chi(g)a$ and $\beta_{g}(b)=\chi(g)b$. Hence
\[
W_{g}\gamma_{\omega}\left(  a^{\ast}\otimes b\right)
=\gamma_{\omega}\left(\alpha_{g}(a)^{\ast}\otimes\beta_{g}(b)\right)
=\gamma_{\omega}\left( a^{\ast}\otimes b\right)
\]
since $\left|  \chi\right|  =1$. Therefore
$\gamma_{\omega}\left(a^{\ast}\otimes b\right)  =c\Omega_{\omega}$
for some $c\in\mathbb{C}$ by \cite[Theorem 4.3.17]{BR} (which uses
above mentioned $G$-abelianness). So
\[
c\gamma_{\mu}(a)=c\pi_{\mu}(a)\Omega_{\omega}=\pi_{\omega}\left(
a\otimes1_{B}\right)  \pi_{\omega}\left(  a^{\ast}\otimes b\right)
\Omega_{\omega}=\gamma_{\omega}\left(  (aa^{\ast})\otimes b\right)
=d\gamma_{\nu}(b)
\]
for some $d\in\mathbb{C}\backslash\{0\}$, since
$\alpha_{g}(aa^{\ast})=\left| \chi(g)\right|
^{2}aa^{\ast}=aa^{\ast}\neq0$ and $\mathbf{A}$ is ergodic. We
conclude that $\gamma_{\mu}(a)$ and $\gamma_{\nu}(b)$ are
proportional, and therefore the eigenvectors of $U$ and $V$ span the
same Hilbert subspace $H_{0}$ of $H_{\omega}$. Lastly, for any $x\in
H_{0}$, we have $U_{g}x=W_{g}x=V_{g}x$ by \cite[Construction
2.3]{D}.
\end{proof}

That some form of asymptotic abelianness should be necessary is
perhaps not surprising (see \cite[Remark 2.7]{S}), however it would
probably be desirable to rather have a version of Proposition 4.6
for C*-dynamical systems (with an invariant state).

\section{Strong mixing}

Throughout this section we consider the situation in Definition 2.2, but with
$G=\mathbb{Z}$. Remember that as in the special case in Section 3,
$\mathbf{A}$ is \textit{strongly mixing} when
\[
\lim_{n\rightarrow\infty}\mu\left(  \alpha_{n}(a)b\right)  =\mu(a)\mu(b)
\]
for all $a,b\in A$. Let
$\mathbf{\tilde{A}=}\left(\tilde{A},\tilde{\mu},\tilde{\alpha}\right)$
be the ``mirror image'' of $\mathbf{A}$ which we referred to after
Definition 2.2. It turns out that
$\tilde{\alpha}_{n}(b)=U_{n}bU_{n}^{\ast}$ for all $b\in\tilde{A}$
with $U$ as in Definition 2.5; see \cite[Construction 3.4]{D} for
more details. Then one can define a joining $\Delta_{n}$ of
$\mathbf{A}$ and $\mathbf{\tilde{A}}$ for every $n$ by
\[
\Delta_{n}(a\otimes b):=\mu_{\bigtriangleup}\left(  \alpha_{n}(a)\otimes
b\right)
\]
for all $a\in A$ and $b\in\tilde{A}$. It is easy to verify that $\Delta_{n}$
is indeed a joining, and in particular $\mu_{\bigtriangleup}=\Delta_{0}$ is a
joining. This joining is an example of what in measure theoretic ergodic
theory is called a \textit{graph joining} (see for example \cite[Examples
6.3]{G} or \cite[Section 2.2]{dlR}). We then have the following simple joining
characterization of strong mixing:

\begin{proposition}
The system $\mathbf{A}$ is strongly mixing if and only if
\begin{equation}
\lim_{n\rightarrow\infty}\Delta_{n}(a\otimes b)=\mu(a)\tilde{\mu}(b)
\tag{5.1}
\end{equation}
for all $a\in A$ and $b\in\tilde{A}$.
\end{proposition}

\begin{proof}
The system $\mathbf{A}$ is strongly mixing if and only if
$\lim_{n\rightarrow \infty}\left\langle U_{n}x,y\right\rangle
=\left\langle x,\Omega\right\rangle \left\langle
\Omega,y\right\rangle $ for all $x,y\in H$, but in turn this is
equivalent to (5.1), since $\tilde{A}\Omega$ is dense in $H$.
\end{proof}

We can also view (5.1) as saying that the sequence $\left(  \Delta_{n}\right)
$ of joinings converges pointwise to the joining $\mu\odot\tilde{\mu}$.

Next we are going to use this result to prove a version of
Ornstein's criterion for strong mixing (in the measure theoretic
setting) \cite[Theorem 2.1]{O} for W*-dynamical systems. Its worth
mentioning that although Ornstein's paper \cite{O} doesn't
explicitly deal with joinings, it did lead to Rudolph's seminal work
\cite{Rud} on joinings and both papers have been very influential in
further developments in classical ergodic theory.

But first we need the following:

\begin{lemma}
Consider a system $\mathbf{A}$ which is not weakly mixing, but with
$\left( A,\mu,\alpha_{n}\right)  $ ergodic for every
$n\in\mathbb{N}$ (the action of $\mathbb{Z}$ in this case is given
by $j\mapsto\left(  \alpha_{n}\right) ^{j}$). Then for every $k>0$
there exists a projection $P\in A$, left fixed by the modular
automorphism group associated with $\mu$, such that $0<\mu(P)<1/k$,
$P\alpha_{n}(P)=\alpha_{n}(P)P$ for all $n$, and
\begin{equation}
\limsup_{n\rightarrow\infty}\mu\left(\alpha_{n}(P)P\right) >k\mu
(P)^{2},\tag{5.2}
\end{equation}
or equivalently,
\begin{equation}
\limsup_{n\rightarrow\infty}\Delta_{n}\left(P\otimes(J\pi(P)J)\right)
>k\left(\mu\odot\tilde{\mu}\right)\left(P\otimes(J\pi(P)J)\right)
\tag{5.3}
\end{equation}
where $J$ is the modular conjugation associated with
$\left(\pi(A),\Omega\right)$.
\end{lemma}

\begin{proof}

The proof is divided into two parts. Part (i) proves the existence
of a projection $P\in A$ such that $0 < \mu (P) < 1/k$,
$P\alpha_n(P) = \alpha_n(P)P$ for all $n$, and (5.2) is satisfied.
Part (ii), for which I am indebted to the referee, proves that the
construction in (i) yields a projection $P$ left fixed by the
modular automorphism group associated with $\mu$, and that this
invariance ensures the equivalence of (5.2) and (5.3).

(i) Using the notation in Definition 2.5, but denoting $U_{1}$
simply as $U$ for simplicity, and correspondingly $\alpha_{1}$ as
$\alpha$, it follows from the fact that $\mathbf{A}$ is ergodic but
not weakly mixing that $U$ has an eigenvalue
$\chi\in\mathbb{C}\backslash\{1\}$ with corresponding eigenvector of
the form $u\Omega$ for some $u\in A$ which means (see \cite[Theorem
2.5] {S}) that $\alpha(u)=\chi u$, where for simplicity of notation
we have identified $A$ with $\pi(A)$ and hence making $\pi$ in
Definition 2.5 the identity mapping $A\rightarrow A$ (we can do this
since $\mu$ is faithful).

Without loss we can assume that $u$ is unitary. Namely
$\alpha(u^{\ast} u)=\bar{\chi}\chi u^{\ast}u=u^{\ast}u$, so
$u^{\ast}u\in\mathbb{C}1$, since $\mathbf{A}$ is ergodic. It follows
that $u^*u=\|u^*u\|1=\|u\|^2 1$, since $u^*u\ge 0$. Since $u\neq0$,
we can assume that $u^{\ast}u=1$ by renaming $u/\left\|  u\right\|
$ as $u$. In the same way ergodicity and this normalization
procedure gives $uu^{\ast}=1$.

Note that $u\notin\mathbb{C}1$, since $\chi\neq1$. Since
$\alpha^{n}(u)=\chi^{n}u$ while $\left(A,\mu,\alpha^{n}\right)$ is
ergodic, it follows that $\chi^{n}\neq1$ for all $n\in\mathbb{Z}$.

Denote the spectrum of $u$ by $\sigma(u)$ and let $E$ be the
spectral measure relative to $\left(  \sigma(u),H\right)  $ with
\[
u=\int\iota dE
\]
where $\iota:\sigma(u)\rightarrow\sigma(u)$ denotes the identity map
(consult \cite[Section 2.5]{M} for a clear exposition of the
spectral theory that we are using here). Note that from the
definition of the spectrum of an element it follows that
$\sigma\left(\alpha^{n}(u)\right)  =\sigma(u)$, hence for
$v\in\sigma(u)$ we have $\chi^{n}v\in\sigma(u)$. But
$\chi^{m}v\neq\chi^{n}v$ and hence
\[
E\left(  \{\chi^{m}v\}\right)  E\left(\{\chi^{n}v\}\right)
=E\left(\{\chi^{m}v\}\cap\{\chi^{n}v\}\right) =0
\]
for any integers $m\neq n$. Setting
\[
\tilde{\chi}:\sigma(u)\rightarrow\sigma(u):v\mapsto\chi v
\]
and defining spectral measures $F:=\alpha\circ E$ and
$D:=E\circ\tilde{\chi }^{-1}$ relative to $\left( \sigma(u),H\right)
$, one can verify that $\int\iota dF=\chi u=\int\iota dD$ and hence
by uniqueness of the spectral measure we have $\alpha\circ
E=E\circ\tilde{\chi}^{-1}$ and more generally
\begin{equation}
\alpha^{n}\circ E=E\circ\tilde{\chi}^{-n}\tag{5.4}
\end{equation}
for all $n\in\mathbb{Z}$. Putting all this together we find that
\[
\alpha^{m}\left(E\left(\{v\}\right)\right)\alpha^{n}\left(E\left(\{v\}\right)\right)
=0
\]
for all integers $m\neq n$, hence
$P_{n}:=\alpha^{1}\left(E\left(\{v\}\right)\right)
+...+\alpha^{n}\left(E\left(\{v\}\right)\right)$ is a projection and
so $0\leq n\mu(\left(E\left( \{v\}\right)\right) \leq1$ for every
$n\in\mathbb{N}$, which means
\begin{equation}
E\left(  \{v\}\right)  =0\tag{5.5}
\end{equation}
for all $v\in\sigma(u)$.

In the remainder of the proof, for any set $V$ in the unit circle we
will simply write $E(V)$ instead of $E\left(  V\cap\sigma(u)\right)
$, and we will also use the notation
$P_{(\theta_{1},\theta_{2}]}:=E\left(  e^{i(\theta
_{1},\theta_{2}]}\right)  $ for any interval
$(\theta_{1},\theta_{2}]$. Consider
$-\pi<\theta_{1}<\theta_{2}\leq\pi$. By (5.4)
$\alpha^{n}\left(P_{(\theta_{1},\theta_{2}]}\right)
=P_{(\theta_{1}+\text{Arg}\chi^{-n}
,\theta_{2}+\text{Arg}\chi^{-n}]}$. But for any $\varepsilon>0$
there are arbitrarily large values of $n$ such that $\left|
\text{Arg}\chi^{-n}\right| <\varepsilon$ and hence such that
$\alpha^{n}\left(  P_{(\theta_{1},\theta _{2}]}\right)
P_{(\theta_{1},\theta_{2}]}\geq P_{(\theta_{1}+\varepsilon
,\theta_{2}-\varepsilon]}$. Furthermore, since $\mu$ is normal while
$\left\langle \Omega,E(\cdot)\Omega\right\rangle $ is a usual
positive measure, one can show that
$\lim_{n\rightarrow\infty}\mu\left(  P_{(\theta
_{1},\theta_{1}+1/n]}\right)  =0$, and by also employing (5.5) one
similarly finds $\lim_{n\rightarrow\infty}\mu\left(
P_{(\theta_{2}-1/n,\theta_{2} ]}\right)  =0$. Combining this with
the fact that $P_{(\theta_{1},\theta_{2}
]}-P_{(\theta_{1}+\varepsilon,\theta_{2}-\varepsilon]}=P_{(\theta_{1}
,\theta_{1}+\varepsilon]}+P_{(\theta_{2}-\varepsilon,\theta_{2}]}$
it follows that for any $\varepsilon^{\prime}$ we can choose
$\varepsilon$ small enough that $\mu\left(
P_{(\theta_{1},\theta_{2}]}-P_{(\theta_{1}+\varepsilon
,\theta_{2}-\varepsilon]}\right)  <\varepsilon^{\prime}$ and
therefore there are arbitrarily large values of $n$ such that
\[
\mu\left(\alpha^{n}\left(P_{(\theta_{1},\theta_{2}]}\right)
P_{(\theta_{1},\theta_{2}]}\right) >\mu\left(
P_{(\theta_{1},\theta_{2} ]}\right)  -\varepsilon^{\prime}.
\]

Now suppose that there is a $\delta>0$ such that
$\mu\left(P_{(\theta _{1},\theta_{2}]}\right)  =0$ or
$\mu\left(P_{(\theta_{1},\theta_{2} ]}\right)  \geq\delta$ for all
$-\pi<\theta_{1}<\theta_{2}\leq\pi$. With
$V(m,r):=(-\pi+2\pi(r-1)/m,-\pi+2\pi r/m]$ we have
$\sum_{r=1}^{m}\mu\left( P_{V(m,r)}\right)  =1$, hence each
\[
\mathcal{I}_{m}:=\left\{  V(m,r):\mu\left( P_{V(m,r)}\right)
\geq\delta ,r\in\{1,...,m\}\right\}
\]
contains at least one element, and we have a sequence of intervals
$I_{m} \in\mathcal{I}_{2^{m}}$ with $I_{m+1}\subset I_{m}$. But then
$\left\langle \Omega,E\left(  \bigcap_{m=1}^{\infty}I_{m}\right)
\Omega\right\rangle \geq\delta$ contradicting (5.5). We conclude
that for any $k^{\prime}>k>0$ there are
$-\pi<\theta_{1}<\theta_{2}\leq\pi$ such that
$0<\mu\left(P_{(\theta_{1},\theta_{2}]}\right) <1/k^{\prime}$. With
$P:=P_{(\theta _{1},\theta_{2}]}$ we have
$P\alpha^{n}(P)=\alpha^{n}(P)P$ from (5.4), completing part (i) of
the proof.

(ii) We continue with the notation in (i).

By \cite[Corollary VIII.1.4]{T} $\alpha\circ\sigma_{t}=\sigma_{t}
\circ\alpha$, where $t\mapsto\sigma_{t}$ is the modular automorphism
group associated with $\mu$. With $u$ and $\chi$ as before, it
follows that $\alpha(\sigma_{t}(u))=\chi\sigma_{t}(u)$. Together
with $\alpha(u)=\chi u$, this implies that
\[
\sigma_{t}(u)=\lambda_{t}u
\]
for some $\lambda_{t}\in\mathbb{C}$, according to \cite[Lemma
2.1(3)]{S}, for every $t\in\mathbb{R}$. Note that $\left|
\lambda_{t}\right|  =1$. From the group property of $\sigma_{t}$ it
is easily verified that $\lambda _{s+t}=\lambda_{s}\lambda_{t}$.
Since $t\mapsto\left\langle x,\sigma _{t}(u)y\right\rangle
=\left\langle \Delta^{-it}x,u\Delta^{-it}y\right\rangle $ is
continuous for all $x,y\in H$, where $\Delta$ is the modular
operator associated with $\left(A,\Omega\right)$, it follows that
$t\mapsto \lambda_{t}$ is continuous. Therefore
\[
\lambda_{t}=e^{i\theta t}%
\]
for all $t\in\mathbb{R}$ for some $\theta\in\mathbb{R}$; see for
example \cite[p. 12]{R}. It follows that
$\Delta^{it}u\Omega=\sigma_{t}(u)\Omega =e^{i\theta t}u\Omega$,
hence by the definition of $J\Delta^{1/2}$ (see for example
\cite[Section 2.5.2]{BR})
\[
Ju^{\ast}\Omega=J\left(  J\Delta^{1/2}\right)  u\Omega=\Delta^{1/2}
u\Omega=e^{\theta/2}u\Omega
\]
and by taking the norm both sides we conclude that $e^{\theta/2}=1$
and therefore $\theta=0$. This proves that
\[
\sigma_{t}(u)=u
\]
for all $t\in\mathbb{R}$.

Note that the fixed point algebra of the modular automorphism group
is itself a von Neumann algebra (as is the fixed point algebra of
any system) and since $u$ is in this fixed point algebra as shown
above, it follows that its spectral projections are too. In
particular
\[
\sigma_{t}(P)=P
\]
for all $t\in\mathbb{R}$. This means that
$\Delta^{it}P\Omega=P\Delta ^{it}\Omega=P\Omega\,$\ and therefore
\[
JP\Omega=J\left(  J\Delta^{1/2}\right)  P^{\ast}\Omega=\Delta^{1/2}
P\Omega=P\Omega
\]
so
\begin{align*}
\Delta_{n}\left(  P\otimes(JPJ)\right)
&=\mu_{\triangle}\left(\alpha
^{n}(P)\otimes(JPJ)\right) =\left\langle
\Omega,\alpha^{n}(P)JPJ\Omega\right\rangle
=\left\langle \Omega,\alpha^{n}(P)P\Omega\right\rangle \\
& =\mu\left(  \alpha^{n}(P)P\right)  .
\end{align*}
Furthermore
\[
\mu\odot\tilde{\mu}\left(  P\otimes(JPJ)\right)  =\mu(P)\left\langle
\Omega,JPJ\Omega\right\rangle =\mu(P)^{2}.
\]
The equivalence of (5.2) and (5.3) now follows.
\end{proof}

Now we are in a position to state and prove our version of
Ornstein's criterion. In its proof we encounter a C*-dynamical
system with invariant state, i.e. a $\left(C,\rho,\tau\right)$ with
$\left(C,\tau\right)$ as in Definition 4.3 and where $\rho$ is any
state on $C$ with $\rho\circ\tau_{g}=\rho$ for all $g\in G$ (with
$G=\mathbb{Z}$ the relevant case). We will refer to such a
$\left(C,\rho,\tau\right)$ as a \textit{C*-dynamical system} as
well. For such a C*-dynamical system \textit{weak mixing} is defined
in the same way as for W*-dynamical systems in Definition 2.5, but
we call it \textit{ergodic} if the fixed point space of the unitary
representation of $\tau$ on the Hilbert space $H$ of the GNS
construction of $\left(  C,\rho\right)  $ is one dimensional, i.e.
$\dim\left\{x\in H:U_{g}x=x\text{ for all }g\in G\right\} =1$.

\begin{theorem}
Let $\mathbf{A}$ be a system such that $\left(
A,\mu,\alpha_{n}\right)  $ is ergodic for every $n\in\mathbb{N}$.
(Alternatively we could assume that $\mathbf{A}$ is weakly mixing.)
Furthermore, assume that there is a real number $k>0$ such that
\[
\limsup_{n\rightarrow\infty}\Delta_{n}(c^{\ast}c)\leq
k\mu\odot\tilde{\mu }(c^{\ast}c)
\]
for all $c\in A\odot\tilde{A}$. Then $\mathbf{A}$ is strongly
mixing.
\end{theorem}

\begin{proof}
Note that $\mathbf{A}$ is weakly mixing, for if it was not, our
assumptions would contradict Lemma 5.2 for $c=P\otimes(J\pi(P)J)$.
In the rest of the proof we only need weak mixing of $\mathbf{A}$,
rather than the ergodicity of $\left( A,\mu,\alpha_{n}\right)$ for
all $n\in\mathbb{N}$.

We now follow the basic argument presented in \cite[Theorem
4.3]{dlR} for the measure theoretic case, and we work in the setting
of the maximal C*-algebraic tensor product as explained at the
beginning of Section 4. Since $J(\mathbf{A},\mathbf{\tilde{A}})$ is
weakly* compact by Proposition 4.2, the sequence $\left(
\Delta_{n}\right)  $ has a cluster point $\omega$ in
$J(\mathbf{A},\mathbf{\tilde{A}})$ in the weak* topology. From our
assumptions it follows that $\omega\leq k\mu\otimes_{m}\tilde{\mu}$.

Note that $\mathbf{\tilde{A}}$ is weakly mixing, since $\mathbf{A}$
is. Also recall that a C*-dynamical system
$\left(C,\rho,\tau\right)$ for an action of $\mathbb{Z}$ is weakly
mixing if and only if
\[
\lim_{N\rightarrow\infty}\frac{1}{N}\sum_{n=1}^{N}\left|  \rho(a\tau
_{n}(b))-\rho(a)\rho(b)\right|  =0
\]
for all $a,b\in C$; see for example \cite[Proposition 5.4]{NSZ} or
\cite[Definition 2.3 and Proposition 3.4]{D2}, and keep in mind that
$\left( \{1,...,N\}\right)  _{N}$ is a F\o lner sequence in
$\mathbb{Z}$. We now use this characterization of weak mixing to
show that the C*-dynamical system
$\mathbf{A\otimes}_{m}\mathbf{\tilde{A}}:=\left( A\mathbf{\otimes}
_{m}\tilde{A},\mu\mathbf{\otimes}_{m}\tilde{\mu},\alpha\mathbf{\otimes}
_{m}\tilde{\alpha}\right)  $ is weakly mixing.\ It will be
convenient to write $\rho:=\mu\mathbf{\otimes}_{m}\tilde{\mu}$ and
$\tau:=\alpha\mathbf{\otimes }_{m}\tilde{\alpha}$. For any
\[
c=\sum_{j=1}^{m}a_{j}\otimes c_{j}\in A\odot\tilde{A}
\]
and
\[
d=\sum_{j=1} ^{m}b_{j}\otimes d_{j}\in A\odot\tilde{A}
\]
we have
\begin{align*}
&  \left|  \rho(c\tau_{n}(d))-\rho(c)\rho(d)\right| \\
&  \leq\sum_{j=1}^{m}\sum_{k=1}^{m}\left|  \mu\left( a_{j}\alpha_{n}
(b_{k})\right)  \tilde{\mu}\left(
c_{j}\tilde{\alpha}_{n}(d_{k})\right) -\mu\left(
a_{j}\alpha_{n}(b_{k})\right)  \tilde{\mu}(c_{j})\tilde{\mu}
(d_{k})\right| \\
&  +\sum_{j=1}^{m}\sum_{k=1}^{m}\left|  \mu\left(  a_{j}\alpha_{n}
(b_{k})\right)  \tilde{\mu}(c_{j})\tilde{\mu}(d_{k})-\mu(a_{j})\mu
(b_{k})\tilde{\mu}(c_{j})\tilde{\mu}(d_{k})\right| \\
&  \leq\sum_{j=1}^{m}\sum_{k=1}^{m}\left\|  a_{j}\right\|  \left\|
b_{k}\right\|  \left|  \tilde{\mu}\left(  c_{j}\tilde{\alpha}_{n}
(d_{k})\right)  -\tilde{\mu}(c_{j})\tilde{\mu}(d_{k})\right| \\
&  +\sum_{j=1}^{m}\sum_{k=1}^{m}\left\|  c_{j}\right\|  \left\|
d_{k}\right\|  \left|  \mu\left(  a_{j}\alpha_{n}(b_{k})\right)
-\mu (a_{j})\mu(b_{k})\right|
\end{align*}
therefore $\lim_{N\rightarrow\infty}\frac{1}{N}\sum_{n=1}^{N}\left|
\rho(c\tau_{n}(d))-\rho(c)\rho(d)\right|  =0$, since $\mathbf{A}$
and $\mathbf{\tilde{A}}$ are both weakly mixing. Now consider
arbitrary $a,b\in A\mathbf{\otimes}_{m}\tilde{A}$ and any
$\varepsilon>0$. Then there are $c,d\in A\odot\tilde{A}$ such that
in the maximal C*-norm $\left\| a-c\right\|  _{m}<\varepsilon$ and
$\left\|  b-d\right\|  _{m}<\varepsilon$, so
\begin{align*}
&  \left|  \rho\left(  a\tau_{n}(b)\right)  -\rho(a)\rho(b)\right| \\
&  \leq\left|  \rho\left(  c\tau_{n}(d)\right)  -\rho(c)\rho(d)\right| \\
&  +\left|  \rho\left(  (a-c)\tau_{n}(b)\right)  \right|  +\left|
\rho\left( c\tau_{n}(b-d)\right)  \right|  +\left|
\rho(c-a)\rho(b)\right|  +\left|
\rho(c)\rho(d-b)\right| \\
&  \leq\left|  \rho\left(  c\tau_{n}(d)\right)
-\rho(c)\rho(d)\right| +2\varepsilon\left\|  b\right\|
_{m}+2\varepsilon\left(  \left\|  a\right\| _{m}+\varepsilon\right)
.
\end{align*}
From all this it follows that
$\lim_{N\rightarrow\infty}\frac{1}{N}\sum _{n=1}^{N}\left|
\rho(a\tau_{n}(b))-\rho(a)\rho(b)\right|  =0$, i.e.
$\mathbf{A\otimes}_{m}\mathbf{\tilde{A}}$ is weakly mixing.

From the definitions of weak mixing and ergodicity of a C*-dynamical
system, it follows that $\mathbf{A\otimes}_{m}\mathbf{\tilde{A}}$ is
ergodic and therefore $\mu\mathbf{\otimes}_{m}\tilde{\mu}\in
E_{\alpha\mathbf{\otimes} _{m}\tilde{\alpha}}$ by \cite[Theorem
4.3.17]{BR} and Definition 4.3. However, if $\omega_{1}\leq
k\omega_{0}$ where $\omega_{0}\in
E_{\alpha\mathbf{\otimes}_{m}\tilde{\alpha}}$ while $\omega_{1}$ is
an invariant state on $A\mathbf{\otimes} _{m}\tilde{A}$ under
$\alpha\otimes_{m}\tilde{\alpha}$, then $\omega _{1}=\omega_{0}$,
since if this was not the case, then $k>1$ so $\omega
_{2}:=(k\omega_{0}-\omega_{1})/(k-1)$ is an invariant state which
gives $\omega_{0}=\omega_{1}/k+(k-1)\omega_{2}/k$ contradicting
$\omega_{0}\in E_{\alpha\mathbf{\otimes}_{m}\tilde{\alpha}}$. So
$\omega=\mu\otimes_{m}\tilde{\mu}$ which means that $\mu
\otimes_{m}\tilde{\mu}$ is the unique cluster point of $\left(\Delta
_{n}\right)$ in the weak* topology. Hence
\[
\operatorname{w*-lim}_{n\rightarrow\infty}\Delta_{n}=\mu\otimes_{m}\tilde{\mu}
\]
in $J(\mathbf{A},\mathbf{\tilde{A}})$. Therefore $\mathbf{A}$ is
strongly mixing by Proposition 5.1.
\end{proof}

Note that the following partial converse is of course also true,
namely if $\mathbf{A}$ is strongly mixing, then it is weakly mixing
and $\operatorname{w*-lim}_{n\rightarrow\infty}\Delta
_{n}=\mu\otimes_{m}\tilde{\mu}$.

\section*{Acknowledgements}
I thank Conrad Beyers and Anton Str\"{o}h for useful conversations.
I also thank the referee for a careful reading of the manuscript and
several constructive suggestions which improved the paper, in
particular in Sections 4 and 5. This work was partially supported by
the National Research Foundation of South Africa.


\begin{thebibliography}{10}                                                                                                %

\bibitem {A}D. Avitzour, \emph{Noncommutative topological dynamics. II}, Trans. Amer.
Math. Soc. \textbf{282} (1984), 121--135. MR 0728705

\bibitem {BDS}C. Beyers, R. Duvenhage and A. Str\"{o}h, \emph{The Szemer\'{e}di
property in ergodic W*-dynamical systems}, J. Operator Theory
\textbf{64} (2010), 35--67. arXiv:0709.1557 [math.OA]

\bibitem {BR}O. Bratteli and D. W. Robinson, \emph{Operator algebras and quantum
statistical mechanics 1}, second edition, Springer-Verlag, New York,
1987. MR 0887100

\bibitem {C}F. Combes, \emph{Sur les faces d'une C*-alg\`ebre}, Bull. Sci. Math. (2) \textbf{93}
(1969), 37--62. MR 0265947

\bibitem {dlR}T. de la Rue, \emph{An introduction to joinings in ergodic theory},
Discrete Contin. Dyn. Syst. \textbf{15} (2006), 121--142. MR 2191388

\bibitem {D}R. Duvenhage, \emph{Joinings of W*-dynamical systems}, J. Math. Anal.
Appl. \textbf{343} (2008), 175--181. MR 2409467

\bibitem {D2}R. Duvenhage, \emph{Bergelson's theorem for weakly mixing
C*-dynamical systems}, Studia Math. \textbf{192} (2009), 235--257.
MR 2504840

\bibitem {Fid}F. Fidaleo, \emph{An ergodic theorem
for quantum diagonal measures}, Infin. Dimens. Anal. Quantum Probab.
Relat. Top. \textbf{12} (2009), 307--320. MR 2541399

\bibitem {F67}H. Furstenberg, \emph{Disjointness in ergodic theory, minimal sets,
and a problem in Diophantine approximation}, Math. Systems Theory
\textbf{1} (1967), 1--49. MR 0213508

\bibitem {G}E. Glasner, \emph{Ergodic theory via joinings}, Mathematical Surveys and
Monographs 101, American Mathematical Society, Providence, RI, 2003.
MR 1958753

\bibitem {H}P. R. Halmos, \emph{On automorphisms of compact groups}, Bull. Amer.
Math. Soc. \textbf{49} (1943), 619--624. MR 0008647

\bibitem {K}U. Krengel, \emph{Ergodic theorems}, Walter de Gruyter \& Co., Berlin,
1985. MR 0797411

\bibitem {KV3}J. Kustermans and S. Vaes, \emph{Locally compact quantum groups}, Ann.
Sci. \'{E}cole Norm. Sup. (4) \textbf{33} (2000), 837--934. MR
1832993

\bibitem {KV4}J. Kustermans and S. Vaes, \emph{Locally compact quantum groups in the
von Neumann algebraic setting}, Math. Scand. \textbf{92} (2003),
68--92. MR 1951446

\bibitem {L}M. Lema\'{n}czyk, \emph{Introduction to ergodic theory from the point of
view of spectral theory}, Lecture Notes on the Tenth Kaisk
Mathematics Workshop, Geon Ho Choe (Ed.), Korea Advanced Institute
of Science and Technology, Math. Res. Center, Taejon, Korea, 1995.

\bibitem {M}G. J. Murphy, \emph{C*-algebras and operator theory}, Academic Press,
Inc., Boston, MA, 1990. MR 1074574

\bibitem {NSZ}C. P. Niculescu, A. Str\"{o}h and L. Zsid\'{o}, \emph{Noncommutative
extensions of classical and multiple recurrence theorems}, J.
Operator Theory \textbf{50} (2003), 3--52. MR 2015017

\bibitem {O}D. S. Ornstein, \emph{On the root problem in ergodic theory}, in:
Proceedings of the Sixth Berkeley Symposium on Mathematical
Statistics and Probability (Univ. California, Berkeley, Calif.,
1970/1971), Vol. II: Probability theory, pp. 347--356, University of
California Press, Berkeley, 1972. MR 0399415

\bibitem {P}K. Petersen, \emph{Ergodic theory}, Cambridge University Press,
Cambridge, 1983. MR 1073173

\bibitem {R}W. Rudin, \emph{Fourier analysis on groups},
Interscience Tracts in Pure and Applied Mathematics, No. 12
Interscience Publishers (a division of John Wiley and Sons), New
York-London, 1962. MR 0152834

\bibitem {Rud}D. J. Rudolph, \emph{An example of a measure preserving map with
minimal self-joinings, and applications}, J. Analyse Math.
\textbf{35} (1979), 97--122. MR 0555301

\bibitem {ST}J.-L. Sauvageot and J.-P. Thouvenot, \emph{Une nouvelle d\'{e}finition de
l'entropie dynamique des syst\`{e}mes non commutatifs}, Comm. Math.
Phys. \textbf{145} (1992), 411--423. MR 1162806

\bibitem {S}E. St\o rmer, \emph{Spectra of ergodic transformations}, J. Funct. Anal.
\textbf{15} (1974), 202--215. MR 0377544

\bibitem {St}\c{S}. Str\u{a}til\u{a}, \emph{Modular theory in operator algebras},
Editura Academiei Republicii Socialiste Rom\^{a}nia, Bucharest,
Abacus Press, Tunbridge Wells, 1981. MR 0696172

\bibitem {T}M. Takesaki, \emph{Theory of operator algebras. II},
Encyclopaedia of Mathematical Sciences, 125. Operator Algebras and
Non-commutative Geometry, 6. Springer-Verlag, Berlin, 2003. MR
1943006

\bibitem {VPhD}S. Vaes, \emph{Locally compact quantum groups}, PhD thesis, Katholieke
Universiteit Leuven, 2001.

\bibitem {vD}A. van Daele, \emph{Locally compact quantum groups. A von Neumann
algebra approach}, arXiv:math/0602212v1 [math.OA].
\end{thebibliography}
\end{document}